\documentclass[10pt, reqno]{amsart}

\usepackage[colorlinks=true, linkcolor=blue!50!green, urlcolor=blue!50!red]{hyperref}
\usepackage{cite, xcolor}
\usepackage{amsmath, graphicx}
\usepackage{mathtools, ulem}
\usepackage{amsthm}
\usepackage{amssymb}
\usepackage{amsfonts}
\usepackage{latexsym}
\usepackage{setspace}
\usepackage{enumitem}
\setlist[itemize]{leftmargin=*}
\setlist[enumerate]{leftmargin=*}
\setlist[enumerate, 1]{label=(\alph*), ref=\alph*}
\setlist[enumerate, 2]{label=(\roman*)}
\usepackage[font=footnotesize,labelfont=bf]{caption}

\renewcommand{\epsilon}{\varepsilon}
\renewcommand{\Pi}{\P}

\usepackage{seqsplit}
\makeatletter
\def\@bignumber#1#2{%
	\ifx#2\end
	#1\let\next\@gobble
	\else
	#1\hspace{0pt plus 1pt}\let\next\@bignumber
	\fi
	\next#2}
\newcommand{\bignumber}[1]{\@bignumber#1\end}
\makeatother

\newtheorem{theorem}{Theorem}

\newtheorem{lemma}[theorem]{Lemma}
\theoremstyle{definition}
\newtheorem{problem}[theorem]{Problem}

\newtheorem{example}[theorem]{Example}

\numberwithin{equation}{section}

\renewcommand{\phi}{\varphi}

\newcommand{\N}{\mathbb{N}}
\renewcommand{\P}{\mathbb{P}}

\newcommand{\R}{\mathbb{R}}

\renewcommand{\S}{\mathbb{S}}
\newcommand{\Z}{\mathbb{Z}}
\renewcommand{\pmod}[1]{\;(\operatorname{mod} #1)}

\renewcommand{\geq}{\geqslant}
\renewcommand{\leq}{\leqslant}

\let\oldenumerate=\enumerate
\def\enumerate{
	\oldenumerate
	\setlength{\itemsep}{5pt}
}
\let\olditemize=\itemize
\def\itemize{
	\olditemize
	\setlength{\itemsep}{5pt}
}

\allowdisplaybreaks

\title[Multidimensional Schinzel-type theorems]{Multidimensional Schinzel-type theorems for multiplicative functions
near prime tuples}

\author[S.R.~Garcia]{Stephan Ramon Garcia}
\address{Department of Mathematics\\Pomona College\\610 N. College Ave., Claremont, CA 91711} 
\email{stephan.garcia@pomona.edu}
\urladdr{http://pages.pomona.edu/~sg064747}
\thanks{Partially supported by NSF grant DMS-1265973.}

	\author{Gabe Udell}
	\email{grua2017@mymail.pomona.edu}
	
	\author{Jiahui Yu}
	\email{jyad2018@mymail.pomona.edu}

\begin{document}

\begin{abstract}
Assuming Dickson's conjecture, we obtain multidimensional analogues of recent results on the behavior of certain multiplicative arithmetic functions near twin-prime arguments.  This is inspired by analogous unconditional theorems of Schinzel undertaken without primality assumptions. 
\end{abstract}

\maketitle

\section{Introduction}
Our aim is to generalize recent results on 
the behavior of certain multiplicative functions near twin-prime arguments
and also several related theorems of Schinzel undertaken without primality assumptions.  
In particular, we obtain multidimensional Schinzel-type results
for more general multiplicative functions, in which prime pairs are replaced with
prime tuples and the additive offsets from the prime arguments are essentially arbitrary. Consequently,
the present work subsumes and generalizes many results from \cite{PRBTP, PRBTP2, GLS}.

Despite a flurry of recent activity \cite{Maynard, Polymath, Polymath2, Zhang},
the existence of infinitely many twin primes is still conjectural.  Consequently, 
results involving twin primes and, more generally,
prime tuples, must rely on unproven conjectures.
Dickson's conjecture is one of the weakest widely-believed conjectures that implies the twin prime
conjecture \cite{Dickson, 1CRTA, Ribenboim}.  It is far weaker than the celebrated Bateman--Horn
conjecture, which concerns polynomials of arbitrary degree and makes
asymptotic predictions \cite{Bateman,Bateman2,1CRTA}.

\medskip\noindent\textbf{Dickson's Conjecture.}  {\it If
  $f_1,f_2,\ldots,f_k \in \Z[t]$ are linear polynomials with positive
  leading coefficients and $f=f_1f_2\cdots f_k$ does not vanish
  identically modulo any prime, then $f_1(t), f_2(t),\ldots,f_k(t)$
  are simultaneously prime infinitely often.}  \medskip

Before stating our main results, we briefly survey some of the relevant literature.
In what follows, $\phi$ denotes the Euler totient function.
In 2017, Garcia, Kahoro, and Luca showed that the Bateman--Horn conjecture implies $\phi(p-1) \geq \phi(p+1)$ for a majority of twin-primes pairs $p,p+2$ and that the reverse inequality holds for a small positive proportion of the twin primes
\cite{PRBTP}.  This bias disappears if only $p$ is assumed to be prime \cite{GL}.
Analogues for prime pairs were obtained in 2018 by Garcia, Luca, and Schaaff \cite{GLS}.
Although preliminary numerical evidence suggested that $\phi(p+1)/\phi(p-1)$ might remain bounded as
$p,p+2$ runs over the twin primes, 
Garcia, Luca, Shi, and Udell proved that Dickson's conjecture implies
that these quotients are dense in $[0,\infty)$ \cite{PRBTP2}. 

The motivation for multidimensional generalizations of these results goes back to Schinzel, who obtained
many similar results without primality restrictions.  For example, \cite[Thm.~1]{Schinzel3} ensures that
\begin{equation}\label{eq:Schinzel}\small
\text{$\left\{ \left( \frac{\phi(n+1)}{\phi(n+2)},\frac{\phi(n+2)}{\phi(n+3)},\ldots, \frac{\phi(n+d)}{\phi(n+d-1)} \right)
: n \in \N \right\}$
 is dense in $[0,\infty)^d$}.
\end{equation}
The same result holds with the sum-of-divisors function $\sigma$ in place of $\phi$ 
(Schinzel quips ``Theorem 2 is obtained from Theorem 1 by replacing the letter $\phi$ with $\sigma$'').  
The seminal result in this direction is Schinzel's 1954 observation that
\begin{equation}\label{eq:SchinzelOld}
  \text{$\left\{ \frac{\phi(n+1)}{\phi(n)} :  n =1,2,\ldots \right\}$ is dense in $[0,\infty)$},
\end{equation}
a variant of an obscure result of Somayajulu \cite{Somayajulu}.
This density result inspired later work of Schinzel, Sierpi\'{n}ski, Erd\H{o}s,
and others \cite{Schinzel2,Schinzel3, SS,SW,Erdos,Erdos2} (see also \cite[Ch.~1]{Sandor}).

Before stating our main result, we require a few words about notation.
In what follows, $\N = \{1,2,\ldots\}$ denotes the set of natural numbers, $\Z$ the set of integers,
and $\R$ the set of real numbers.  We let
$\P = \{2,3,5,7,11,\ldots\}$ denote the set of prime numbers; the symbol $p$ always refers to 
a prime number.  An $m$-tuple $(\alpha_1,\alpha_2,\ldots,\alpha_m) \in \Z$ is {\it admissible} if
there does not exist a $p \in \P$ such that $\alpha_1,\alpha_2,\ldots,\alpha_m$
form a complete residue system modulo $p$.  This ensures that no congruence obstruction
prevents the linear polynomials $x-\alpha_1, x-\alpha_2, \ldots, x - \alpha_m$ from being simultaneously
prime infinitely often.

Our main theorem is both a broad multidimensional generalization of the results of \cite{PRBTP2}
and a version of Schinzel's theorem \eqref{eq:Schinzel} with primality restrictions.

\begin{theorem}\label{Theorem:Main}
Let $f$ be a positive multiplicative function such that
\begin{enumerate}
    \item $\lim_{p\to \infty} f(p)=1$, and
    \item $\prod_{p\in \P}f(p)$ is not absolutely convergent,
\end{enumerate} 
let $$\lim_{n\to  \infty}\frac{h(n+1)}{h(n)}=\kappa \in (0,\infty),$$ and 
let $g = fh$.  For any distinct $\alpha_1,\alpha_2,\ldots, \alpha_{d}\in \Z$ and admissible 
$(\beta_1, \beta_2, \ldots, \beta_m)$ with $\alpha_i\neq \beta_j$ 
 for $1 \leq i \leq d$ and $1 \leq j \leq m$, Dickson's conjecture implies both
\begin{equation}\label{eq:S2}\small
\left\{\left(\frac{g(n+\alpha_2)}{g(n+\alpha_1)}, \frac{g(n+\alpha_3)}{g(n+\alpha_1)}, \dots, \frac{g(n+\alpha_{d})}{g(n+\alpha_1)}\right): n+\beta_1,n+\beta_2,\ldots, n+\beta_m \in \P\right\}    
\end{equation}
and 
\begin{equation}\label{eq:S1}\small
\left\{\left(\frac{g(n+\alpha_1)}{g(n+\alpha_2)}, \frac{g(n+\alpha_2)}{g(n+\alpha_3)}, \dots, \frac{g(n+\alpha_{d-1})}{g(n+\alpha_{d})}\right): n+\beta_1,n+\beta_2,\ldots, n+\beta_m \in \P\right\}
\end{equation}
are dense in $[0,\infty)^{d-1}$.
\end{theorem}

The main ingredient in the proof of Theorem \ref{Theorem:Main} is the following 
result, which is of independent interest (despite its more technical statement)
since it generalizes several results from \cite{PRBTP2} that do not fall
under the umbrella of Theorem \ref{Theorem:Main}.

\begin{theorem}\label{Theorem:Primary} 
Let $f$ be a positive multiplicative function such that
\begin{enumerate}
    \item $\lim_{p\to \infty} f(p)=1$, and
    \item  for some subset $\S\subseteq \P$, $\prod_{p\in \S}f(p)$  diverges to $0$.
\end{enumerate}
For distinct $\alpha_1,\alpha_2,\ldots, \alpha_d \in \Z$ and an admissible 
$(\beta_1, \beta_2, \ldots, \beta_m) \in \Z^m$ with $\alpha_i\neq \beta_j$ for all $i,j$,
Dickson's conjecture implies that
\begin{equation*}
\big\{ \big(f(n+\alpha_1), f(n+\alpha_2), \dots, f(n+\alpha_d) \big): 
n+\beta_1,n+\beta_2,\ldots, n+\beta_m \in \P \big\}    
\end{equation*}
is dense in $[0,r]^d$, in which
\begin{equation}\label{eq:r}
r=\min{\bigg\{f\bigg(\prod_{j=1}^{\pi(m+d)}p_j^{x_{j}}\bigg): 0 \leq x_{j} \leq \lfloor \log_2 d \rfloor +1 \bigg\}}.
\end{equation}
Here $\pi(x)= \sum_{p\leq x} 1$ denotes the prime-counting function.
\end{theorem}

This paper is organized as follows.  In Section \ref{Section:Applications} we present a wide array of examples
and applications of Theorems \ref{Theorem:Main} and \ref{Theorem:Primary}.
We present several necessary lemmas in Section \ref{Section:Preliminaries}
before proceeding to Section \ref{Section:Primary},
which concerns the proof of Theorem \ref{Theorem:Primary}.
The proof of Theorem \ref{Theorem:Main} is contained in Section \ref{Section:Proof}.  We conclude with several remarks and open problems in Section \ref{Section:Further}.

\section{Examples and Applications}\label{Section:Applications}
We demonstrate that a wide variety of known and novel
results follow from Theorems \ref{Theorem:Main} and \ref{Theorem:Primary}. 
Since there are so many consequences of these theorems, we split the following list of examples
into one-dimensional and multidimensional categories.  In particular, we highlight some striking numerical examples
which illustrate that our method of proof narrows down the search for suitable prime tuples to the extent that
the relevant computations are feasible on a standard laptop computer.

\subsection{One-dimensional results}
Before we recover all of the main one-dimensional results from \cite{PRBTP2}, 
we first direct the reader to Table \ref{Table:1d}, which contains a few curious examples.
As usual, we assume the truth of Dickson's conjecture.

\begin{table}\small
\begin{equation*}
\begin{array}{c|c|c|c}
\xi	& a_1,a_2 & p & \frac{\phi(p+a_2)}{\phi(p+a_1)} \\ 
	\hline 
	\gamma & -1, 1 & 95674157816864951038010948990752780001 & \underline{0.57721566490153}0\ldots  \\[2pt]
	\pi & 5, 16 & 12029840180666026511494250079901  & \underline{3.1415926535}5768\ldots \\[2pt]
	e & 11, 16 & 106784808714334981809995191 & \underline{2.71828182}788915\ldots  \\	
\end{array}
\end{equation*}
\caption{Here $p,p+6,p+12,p+18$ are prime and $\phi(p+a_2)/\phi(p+a_1)$ closely approximates
a fundamental mathematical constant.  Underlined digits agree with those of the constant in question.}
\label{Table:1d}
\end{table}

\begin{example}
Theorem \ref{Theorem:Main} with $f(n) = \phi(n)/n$, $h(n)=n$, $\alpha_1 = -1$, $\alpha_2 = 1$, $\beta_1=0$, and $\beta_2 = 2$,
implies \cite[Thm.~1]{PRBTP2}:
\begin{equation*}
\left\{ \frac{\phi(p+1)}{\phi(p-1)} : p,p+2 \in \P\right\}\quad
\text{is dense in $[0,\infty)$}.
\end{equation*}
\end{example}

\begin{example}
Apply Theorem \ref{Theorem:Main} to $f(n) = \sigma(n)/n$, 
 in which $\sigma(n) = \sum_{d|n} d$,
with the same $h$, $\alpha_i$, and $\beta_i$, as in 
the previous example, and obtain \cite[Thm.~4a]{PRBTP2}: 
\begin{equation*}
\left\{ \frac{\sigma(p+1)}{\sigma(p-1)} : p,p+2 \in \P\right\}\quad
\text{is dense in $[0,\infty)$}.
\end{equation*}
\end{example}

\begin{example}\label{Example:Bad}
Since $\limsup_{p\to\infty} \phi(p+1) / \phi(p) = \frac{1}{2}$,
we do not expect a straightforward prime version of Schinzel's result \eqref{eq:SchinzelOld}
(it is known unconditionally that $\{ \phi(p+1) / \phi(p) : p \in \P\}$ is dense in $[0,\frac{1}{2}]$ \cite[Thm.~2]{PRBTP2}).
However, we can obtain shifted Schinzel-type results with primality restrictions, such as
\begin{equation*}
\left\{ \frac{\phi(n+1)}{\phi(n)} : n+7,n+9 \in \P\right\}\quad
\text{is dense in $[0,\infty)$}.
\end{equation*}
\end{example}

\begin{example}
Let $f(n) = \phi(n)/n$ in Theorem \ref{Theorem:Primary} with $\alpha_1 = 1$, $\beta_1=0$, and $\beta_2 = 2$. 
Then $m+d = 3$, $\pi(m+d)=2$, and $\lfloor \log_2 d \rfloor +1 =1$.
Since
\begin{equation*}
r=\min\left\{\dfrac{\phi(2\cdot 3)}{2 \cdot 3},\dfrac{\phi(2)}{2},\dfrac{\phi(3)}{3},\dfrac{1}{1}\right\}=\frac{1}{3},
\end{equation*}
Theorem \ref{Theorem:Primary} implies 
$\{ \phi(p+1)/(p+1): p,p+2 \in \P\}$ is dense in $[0,\frac{1}{3}]$
and hence we recover \cite[Thm.~3]{PRBTP2}:
\begin{equation*}
\left\{\dfrac{\phi(p+1)}{\phi(p)}: p,p+2 \in \P\right\}\quad
\text{is dense in $[0,\tfrac{1}{3}]$}. 
\end{equation*}
\end{example}

\begin{example}
Let $f(n) = n/\sigma(n)$ and use the same parameters as the previous example.  Since
$r= \min\{\frac{6}{\sigma(6)},\frac{2}{\sigma(2)},\frac{3}{\sigma(3)},\frac{1}{1}\}=\frac{1}{2}$,
Theorem \ref{Theorem:Primary} implies
$\{\frac{ p+1}{\sigma(p+1)}: p,p+2 \in \P\}$ is dense in $[0,\tfrac{1}{2}]$
and we recover \cite[Thm.~4c]{PRBTP2}:
\begin{equation*}
\left\{\dfrac{\sigma(p+1)}{\sigma(p)}: p,p+2 \in \P\right\}\quad
\text{is dense in $[2,\infty)$}. 
\end{equation*}
\end{example}

\subsection{Higher-dimensional generalizations}
Theorem \ref{Theorem:Main} permits higher-dimensional generalizations
of the key results of \cite{PRBTP2}.  These are prime analogues of
Schinzel's seminal result \eqref{eq:Schinzel}.  Table \ref{Table:2d} displays a variety of appealing examples.
In what follows,  we assume the truth of Dickson's conjecture.

\begin{table}\footnotesize
	\begin{tabular}{p{1cm}|c|p{1.5cm}|p{1.8cm}|p{1.8cm}|p{1.4cm}|p{1.4cm}}
	$\xi_1$ & $\xi_2$ & $\beta_1,\beta_2,\ldots$ & $\alpha_1,\alpha_2,\alpha_3$ & $p$ & $\frac{\phi(p+a_2)}{\phi(p+a_1)}$ & $\frac{\phi(p+a_3)}{\phi(p+a_1)}$ \\ 	
	\hline 
	$\sqrt{2}$ & $\sqrt{3}$ & 0,2 & 43, 67, 163 & \bignumber{7511844784494160990486495708935278184940965989331896977463504533990177620200653044436211} &\uline{\bignumber{1.414213562373095048801688}}\bignumber{850439}\ldots & \uline{\bignumber{1.73205080756887729352}}\bignumber{8311231013}\ldots \\ 
\hline	
	$\frac{\sqrt{5}+1}{2}$ & $\frac{\sqrt{5}-1}{2}$ & $0,2$ & $5,8,13$ & \seqsplit{13008439144450632672234079283299510995557205376306391}  & \uline{\bignumber{1.6180339887498948}}\bignumber{70}\ldots  & \uline{\bignumber{0.6180339887498948}}\bignumber{557}\ldots   \\ 
	\hline 
	$\int_0^1 x^x\,dx$& $\int_0^1 \frac{1}{x^x}\,dx$ & $0,10, 12, 64$, $88$ & $561, 1105, 1729$ & \bignumber{9188455696503721950121061053683259795883948157898722348949859809}& \uline{\bignumber{0.783430510712134}}\bignumber{508863}\ldots& \uline{\bignumber{1.2912859970626635404}}\bignumber{3969}\ldots \\ 	
\hline
$e/10$&$\pi/10$& 0, 2, 56, 80, 196884 & 314, 159, 265 & \bignumber{9613597587126448065138098030260432768517}  & \uline{\bignumber{0.27182818}}\bignumber{317353332418}\ldots & \uline{\bignumber{0.314159265}}\bignumber{83582612337}\ldots\\ 	
\end{tabular}
\caption{Here $p+\beta_1, p+\beta_2,\ldots$ are prime and $(\frac{\phi(p+a_2)}{\phi(p+a_1)},\frac{\phi(p+a_3)}{\phi(p+a_1)})$ closely approximates
$(\xi_1,\xi_2)$.  Underlined digits agree with those of the constants in question.  Note that
$43$, $67$, and $163$ are the largest Heegner numbers; $5$, $8$, and $13$ are Fibonacci numbers;
$561$, $1105$, and $1729$ are the first three Carmichael numbers; and
$196884$ is the coefficient of $q$ in the Fourier expansion of the $J$-invariant (as in Monstrous Moonshine).
}
\label{Table:2d}
\end{table}

\begin{example}
Apply Theorem \ref{Theorem:Main} to $f(n) = \phi(n)/n$ and deduce that both
\begin{equation*}
\left\{\left(\frac{\phi(n+\alpha_1)}{\phi(n+\alpha_2)}, \frac{\phi(n+\alpha_2)}{\phi(n+\alpha_3)}, \dots, \frac{\phi(n+\alpha_{d-1})}{\phi(n+\alpha_{d})}\right): n+\beta_1,n+\beta_2,\ldots, n+\beta_m \in \P\right\}
\end{equation*}
and
\begin{equation*}
\left\{\left(\frac{\phi(n+\alpha_2)}{\phi(n+\alpha_1)}, \frac{\phi(n+\alpha_3)}{\phi(n+\alpha_1)}, \dots, \frac{\phi(n+\alpha_{d})}{\phi(n+\alpha_1)}\right): n+\beta_1,n+\beta_2,\ldots, n+\beta_m \in \P\right\}
\end{equation*}
are dense in $[0,\infty)^{d-1}$
for any distinct $\alpha_1,\alpha_2,\ldots, \alpha_{d}\in \Z$ and admissible 
$(\beta_1, \beta_2, \ldots, \beta_m)$ with $\alpha_i\neq \beta_j$ 
 for $1 \leq i \leq d$ and $1 \leq j \leq m$.
These are prime analogues of Schinzel's theorem \eqref{eq:Schinzel}.
In a similar manner, these results hold for $\sigma$ as well.
\end{example}

\begin{example}
Let 
\begin{equation*}
f(n) = \exp\bigg(\sum_{ p \in \P} \frac{\nu_{p}(n)}{p} \bigg)
\end{equation*}
and $h(n)=1$, in which $\nu_{p}$ is the $p$-adic valuation.  Then
$f(p) = e^{1/p} \to 1$ and
\begin{equation*}
\prod_{p \in \P}f(p) = \prod_{i=1}^{\infty} \exp\bigg(\frac{1}{p}\bigg)
= \exp\bigg(\sum_{i=1}^\infty \frac{1}{p}\bigg)
\end{equation*}
diverges.  Then Theorem \ref{Theorem:Main} implies
\begin{equation*}
\left\{\left(\frac{f(n+\alpha_1)}{f(n+\alpha_2)}, \frac{f(n+\alpha_2)}{f(n+\alpha_3)}, \ldots, \frac{f(n+\alpha_{d-1})}{f(n+\alpha_{d})}\right): n+\beta_1,n+\beta_2,\ldots, n+\beta_m \in \P \right\}
\end{equation*}
and
\begin{equation*}
\left\{\left(\frac{f(n+\alpha_2)}{f(n+\alpha_1)}, \frac{f(n+\alpha_3)}{f(n+\alpha_1)}, \dots, \frac{f(n+\alpha_{d})}{f(n+\alpha_1)}\right): n+\beta_1,n+\beta_2,\ldots, n+\beta_m \in \P\right\}  
\end{equation*}
are dense in $[0,\infty)^{d-1}$
for any distinct $\alpha_1,\alpha_2,\ldots, \alpha_{d}\in \Z$ and admissible 
$(\beta_1, \beta_2, \ldots, \beta_m)$ with $\alpha_i\neq \beta_j$ 
 for $1 \leq i \leq d$ and $1 \leq j \leq m$.
\end{example}

\section{Preliminaries}\label{Section:Preliminaries}
The following lemma is essentially due to Schinzel \cite[Lem.~1]{Schinzel3}, except that here
we insist upon the extra condition $\ell_k > n_k$ and we consider $0<C<1$ instead of $C>1$.
We provide the proof here because of these modifications.

\begin{lemma}\label{Lemma:Ratio}
    Let $u_n$ denote an infinite sequence of real numbers such that 
    \begin{equation*}
        \lim_{n\to  \infty} u_n = 0 \qquad\text{and}\qquad \lim_{n\to  \infty} \dfrac{u_{n+1}}{u_n}=1.
    \end{equation*}
    For each $0<C<1$ and strictly increasing sequence ${n_k}$ in $\N$, there exists 
    $\ell_k \in \N$ such that 
    \begin{equation*}
        \ell_k > n_k \quad\text{for $k=1,2,3\dots$} \qquad\text{and}\qquad 
        \lim_{k \to  \infty} \dfrac{u_{\ell_k}}{u_{n_k}}=C.
    \end{equation*}
\end{lemma}

\begin{proof}
For $k\in \N$, let $\ell_k \geq n_k$ be the least natural number such that 
$$\frac{u_{l_k}}{u_{n_k}}\leq C.$$ Such a number exists because $\lim_{n\to \infty} u_n = 0$. Furthermore, $\ell_k>n_k$ because $C<u_{n_k} / u_{n_k}=1$ by assumption. 
The minimality of $\ell_k$ ensures that
    \begin{equation}\label{eq:Subtle}
        C< \frac{u_{\ell_k-1}}{u_{n_k}} 
    \end{equation}
    and hence
    \begin{equation*}
        C\dfrac{u_{\ell_k}}{u_{\ell_k-1}}
        = \left(C  \frac{u_{n_k}}{u_{\ell_k-1}}\right) \dfrac{u_{\ell_k}}{u_{n_k}}
        <\dfrac{u_{\ell_k}}{u_{n_k}}\leq C.
    \end{equation*}
    Thus,
    \begin{equation*}
        \lim_{k \to  \infty} \dfrac{u_{\ell_k}}{u_{n_k}}=C. \qedhere
    \end{equation*}
\end{proof} 

The next lemma is a generalization of \cite[Lem.~5]{PRBTP2} (see also \cite[Prop. 8.8]{deKoninck}).

\begin{lemma}\label{Lemma:Finite}
    Let $f$ be a positive multiplicative function such that $\lim_{p \to {\infty}}f(p)=1$ and 
    $\prod_{p \in \S}f(p)$ diverges to zero for some $\S \subset \P$.  
    For any finite subset $\P'\subset \P$, 
    \begin{equation*}
        \{f(n): \text{$n$ squarefree, $p\nmid n$ for all $p\in \P'$}\} \quad 
        \text{is dense in $[0,1]$}.
    \end{equation*}
\end{lemma}

\begin{proof}
    Let $q_i$ denote the $i$th smallest prime in the infinite set $\S \backslash \P'$.  Define
    \begin{equation*}
        u_n= \prod_{i=1}^n f(q_i),
    \end{equation*}
    which tends to zero as $n\to \infty$ and satisfies
    \begin{equation*}
        \lim_{n\to  \infty} \dfrac{u_{n+1}}{u_n}=\lim_{n\to  \infty}f(q_{n+1})=1.
    \end{equation*}
    Let $n_k$ be an increasing sequence in $\N$ and $0<C<1$.
    Lemma \ref{Lemma:Ratio} provides a sequence $\ell_n$ in $\N$ such that
    \begin{equation*}
        \ell_k > n_k \quad\text{for $k=1,2,\ldots$} 
        \qquad\text{and}\qquad 
        \lim_{k \to  \infty} \dfrac{u_{\ell_k}}{u_{n_k}}=C.
    \end{equation*}
    Then $w_k= \prod_{i=n_k+1}^{\ell_k} q_i$ is squarefree, not divisible by any element of $\P'$,
    and satisfies
    \begin{equation*}
        f(w_k)= f\left(  \prod_{i=n_k+1}^{\ell_k} q_i \right) = \prod_{i=n_k+1}^{\ell_k} f(q_i)
        =\dfrac{ \prod_{i=1}^{\ell_k} f(q_i)}{ \prod_{i=1}^{n_k} f(q_i)}=\dfrac {u_{\ell_k}}{u_{n_k}} \to C. \qedhere
    \end{equation*}
\end{proof}

\begin{lemma}\label{Lemma:W}
    Let $f$ be a positive multiplicative function such that $\lim_{p \to {\infty}}f(p)=1$ and 
    $\prod_{p \in \S}f(p)$ diverges to zero for some $\S \subset \P$.  
    For any finite  $\P'\subset \P$, the set of $d$-tuples
    $(f(w_1),f(w_2),...,f(w_d))$ such that
    \begin{enumerate}
        \item $w_1,w_2,\ldots,w_d \in \N$ are squarefree and pairwise relatively prime, and
        \item $p\nmid w_i$ for $p \in \P'$ and $1 \leq i \leq d$,
    \end{enumerate}
    is dense in $[0,1]^d$
\end{lemma}

\begin{proof}
    We proceed by induction on $d$.
    If $I_1 \subset [0,1]$ is an open interval,
    Lemma \ref{Lemma:Finite} provides a
    squarefree $w_1$ such that $f(w_1) \in I_1$ and $p \nmid w_1$ for all $p \in \P'$.
    Let $I_1, I_2,\dots,I_d \subset [0,1]$ be open intervals and suppose that 
    there are squarefree, pairwise relatively prime $w_1,w_2,\ldots,w_{d-1}$ such that
    $p\nmid w_i$ for all $p\in \P'$ and $f(w_i) \in I_i$ for $1 \leq i \leq d-1$.
    Let $\P''$ be the union of $\P'$ with the set of divisors of $w_1,w_2,\ldots, w_{d-1}$. 
    Lemma \ref{Lemma:Finite} provides a squarefree $w_d$, coprime to $w_1,w_1,\ldots,w_{d-1}$,
    such that $(f(w_1),f(w_2),\dots, f(w_d))\in I_1\times I_2\times\cdots \times I_d$
    and $p\nmid w_d$ for each $p\in \P''\supset\P'$.
    This concludes the induction.
\end{proof}

The next lemma provides a simple method to pass between results about sets of the form \eqref{eq:S1}
and \eqref{eq:S2}.

\begin{lemma}\label{Lemma:Homeo}
The function $\Phi:(0,\infty)^{d-1}\to (0,\infty)^{d-1}$ defined by 
\begin{equation*}
\Phi(y_1,y_2,\dots,y_{d-1}) = \bigg(\frac{1}{y_1},\frac{y_1}{y_2},\frac{y_2}{y_3},\dots, \frac{y_{d-2}}{y_{d-1}} \bigg)
\end{equation*}
is a homeomorphism.
\end{lemma}

\begin{proof}
Since $\Phi$ is continuous, it suffices to observe that
the continuous function $\Psi:(0,\infty)^{d-1} \to (0,\infty)^{d-1}$ 
\begin{equation*}
\Psi(x_1,x_2,\dots,x_{d-1}) = \bigg(\frac{1}{x_1},\frac{1}{x_1 x_2},\frac{1}{x_1 x_2 x_3},\ldots, \frac{1}{\prod_{i=1}^{d-1} x_i} \bigg)
\end{equation*}
inverts $\Phi$.
\end{proof}


\section{Proof of Theorem \ref{Theorem:Primary}}\label{Section:Primary}
We break the proof of Theorem \ref{Theorem:Primary}
into a number of subsections for clarity.  This organization highlights the particular
parameters involved at each stage.

\subsection{Initial Setup and Outline}\label{Subsection:Lem1}
Suppose $f$ is a positive multiplicative function satisfying hypotheses (a) and (b) of Theorem \ref{Theorem:Primary}.
Let $\alpha_1,\alpha_2,\ldots \alpha_d \in \Z$ be distinct, let $(\beta_1,\beta_2,\ldots,\beta_m)$ 
be an admissible $m$-tuple with $\alpha_i\neq \beta_j$ for $1 \leq i \leq d$ and $1 \leq j \leq m$.
Define $L=\pi(m+d)$ and let $r$ be given by \eqref{eq:r}.

It suffices to show that for each $\epsilon >0$ and $\xi=(\xi_1,\xi_2,\dots,\xi_d) \in [0,r]^d$, 
there is an $n \in \N$ such that $n+\beta_j$ is prime for $1 \leq j \leq m$ and 
$f(n+\alpha_i)\in (\xi_i(1-\epsilon),\xi_i(1+\epsilon))$ for each $1 \leq i \leq d$.

\subsection{The integers $b_1,b_2,\ldots, b_L$}
Since $(\beta_1,\beta_2,\dots, \beta_m)$ is an admissible $m$-tuple,
\begin{equation*}
P(t)=(t+\beta_1)(t+\beta_2)\cdots(t+\beta_{m})
\end{equation*}
does not vanish identically modulo any prime.  Consequently, for each $p_j$ with $j=1,2,\ldots,L$, there is some $b_j \in \Z$ 
such that $P(b_j)\not\equiv 0 \pmod{p_j}$ and hence
\begin{equation}\label{defb}
p_j \nmid (b_j+\beta_i) \quad \text{for $1 \leq i\leq m$}. 
\end{equation}

\subsection{The exponents $x_{i,j}$}
Let 
\begin{equation*}
s =\lfloor\log_2 d \rfloor +1
\end{equation*}
and observe that $p_{j}^s \geq 2^s > d$ for each $j \in \N$.  Since there are precisely 
$p_j^s$ multiples of $p_j$ modulo $p_j^{s+1}$, there is an $e_j \in \Z$ such that
\begin{equation*}
    \alpha_i+e_jp_j+b_j \not\equiv 0 \pmod{p_j^{s+1}} \quad \text{for $1 \leq i \leq d$}.
\end{equation*}
Define
\begin{equation}\label{eq:xdef}
x_{i,j}=\max\big\{y\,:\, p_j^y\!\mid\!( \alpha_i+e_jp_j+b_j)\big\}
\end{equation}
and observe that $x_{i,j} \leq s$ for $1 \leq i\leq d$ and $1 \leq j \leq L$.

\subsection{The intervals $I_1,I_2,\ldots,I_d$}
For $i=1,2,\ldots,d$, define
\begin{equation}\label{eq:I}
    I_i = \left(\xi_i\frac{1-\epsilon}{f(\prod_{j=1}^L p_j^{x_{i,j}})},\xi_i\frac{1+\epsilon}{f(\prod_{j=1}^L p_j^{x_{i,j}})} \right) \cap (0,1).
\end{equation}
Since $\xi_i \in [0,r]$ and $0 < r \leq f(\prod_{j=1}^L p_j^{x_{i,j}})$, 
it follows that each $I_i$ is nonempty.

\subsection{The natural numbers $w_1,w_2,\ldots,w_d$}
Define
\begin{equation}\label{eq:PP}\small
\P'=\{p_1, p_2, \dots, p_L\} \cup \bigg\{p \,:\, p\! \bigm|  \!\!
\bigg( \prod_{i=1}^{d}\alpha_i \bigg) \!\bigg( 
 \prod_{\substack{1\leq i\leq d \\ 1\leq j\leq m}}(\alpha_i-\beta_j) \bigg)  \!
\bigg(\prod_{1\leq i<j \leq d} (\alpha_i-\alpha_j)\bigg)
\bigg \}.
\end{equation}
Lemma \ref{Lemma:W} provides pairwise relatively prime $w_1, w_2, \dots, w_d \in \N$ 
such that $f(w_i) \in I_i$ and $p \nmid w_i$ for all $p \in \P'$ and $1 \leq i \leq d$.

\subsection{The natural number $c$}
Since $p_1,p_2,\ldots, p_L, w_1,\ldots, w_d$ are pairwise relatively prime,
the Chinese remainder theorem yields $c \in \N$ such that
\begin{align}
	c & \equiv  e_jp_j+b_j \pmod{p_j^{s+1}}   && \text{for $1\leq j \leq L$}, \label{eq:cclever}\\
	c & \equiv  w_i-\alpha_i \pmod{w_i^2} && \text{for $1 \leq i \leq d$}. \label{eq:ca}
\end{align}

\subsection{The polynomials}
Define
\begin{equation}\label{eq:h0def}
h_0(t)=\bigg[\bigg(\prod_{j=1}^dw_j^2\bigg)\bigg(\prod_{j=1}^{L}p_j^{s+1}\bigg)\bigg] t+c
\end{equation}
and 
\begin{align}
h_i(t)&=h_0(t)+\beta_i && \text{for $1\leq i \leq m$}, \label{eq:hdef} \\
g_i(t)&=\frac{h_0(t)+\alpha_i}{w_i\prod_{j=1}^L p_j^{x_{i,j}}} && \text{for $1\leq i \leq d$}. \label{eq:gdef}
\end{align}

\subsection{Integer coefficients}
By construction, $h_1,h_2,\ldots,h_m \in \Z[t]$.  Let us verify that $g_1,g_2,\ldots,g_d \in \Z[t]$.
From \eqref{eq:h0def} and \eqref{eq:gdef}, we have
\begin{equation}\label{eq:checkg}
g_i(t) = \frac{(\prod_{j=1}^d w_j^2)(\prod_{j=1}^L p_j^{s+1})}{w_i\prod_{j=1}^L p_j^{x_{i,j}}}t + \frac{c+\alpha_i}{w_i\prod_{j=1}^L p_j^{x_{i,j}}}.
\end{equation}
The coefficient of $t$ is an integer since $x_{i,j} \leq s$ for $1 \leq i \leq d$ and $ 1\leq j \leq L$.  
For the constant term, first observe that each $w_i \mid (c+\alpha_i)$ by \eqref{eq:ca}.  
The definition \eqref{eq:xdef} of $x_{i,j}$ ensures that
$p_j^{x_{i,j}} \mid (\alpha_i+e_jp_j+b_j)$ and 
\eqref{eq:cclever} implies
\begin{equation}\label{cWithAlpha}
    \alpha_i+e_jp_j+b_j  \equiv \alpha_i + c \pmod{p_j^{s+1}}.
\end{equation}
Consequently, $\prod_{j=1}^L p_j^{x_{i,j}} \mid (c+\alpha_i)$. 
Since $p_1,p_2,\ldots, p_L, w_1,\ldots, w_d$ are pairwise relatively prime,
the constant term in \eqref{eq:checkg} is an integer.
Thus, $g_1,g_2,\ldots,g_d\in \Z[t]$.

\subsection{Nonvanishing modulo small primes}
Consider 
\begin{equation}\label{eq:F}
F(t)=\bigg(\prod_{i=1}^{m} h_i(t)\bigg)\bigg(\prod_{i=1}^d g_i(t)\bigg) \in \Z[t]
\end{equation}
(the first product excludes $h_0$)
and observe that $\deg F = m+d$.  
We claim that $F$ does not vanish modulo any of $p_1,p_2,\ldots,p_L$.
Since $x_{i,\ell}\leq s$,
\begin{equation*}
p_{\ell} \Bigm| \frac{\left(\prod_{j=1}^{d}w_j^2\right)\left(\prod_{j=1}^L p_j^{s+1}\right) }{w_i\prod_{j=1}^L p_j^{x_{i,j}}}\end{equation*}
and hence $h_0(t)\equiv c \pmod{p_{\ell}}$ by \eqref{eq:h0def}.
The definition \eqref{eq:xdef} of $x_{i,\ell}$ and \eqref{cWithAlpha} imply
\begin{equation*}
x_{i,\ell} = \max\{y\,:\, p_{\ell}^y\!\mid\! (c+\alpha_i)\}
\end{equation*}
and therefore the constant term in \eqref{eq:checkg} is not divisible by ${p_{\ell}}$.
Thus,
\begin{equation*}
    g_i(t)=\frac{h_0(t)+\alpha_i}{w_i\prod_{j=1}^L p_j^{x_{i,j}}} \equiv \frac{c+\alpha_i}{w_i\prod_{j=1}^L p_j^{x_{i,j}}}
\not\equiv 0 \pmod{p_{\ell}}.
\end{equation*}
Since \eqref{eq:cclever} implies that $c\equiv b_{\ell} \pmod{p_{\ell}}$,
it follows from \eqref{defb} that
\begin{equation*}
    h_i(t) =h_0(t)+\beta_i \equiv  c+\beta_i \equiv b_{\ell}+\beta_i \not\equiv 0 \pmod{p_{\ell}}.
\end{equation*}
Consequently, $F$ does not vanish modulo any of $p_1,p_2,\ldots,p_L$.

\subsection{Nonvanishing modulo large primes}
Suppose toward a contradiction that $F$ vanishes identically modulo some prime $p \notin \{p_1,p_2,\ldots,p_L\}$.
Observe that $p > m+d = \deg F$ since $L = \pi(m+d)$.   
The fully-factored presentation \eqref{eq:F} ensures that some linear factor of $F$ vanishes identically modulo $p$. 

The definitions \eqref{eq:h0def}, \eqref{eq:hdef}, and \eqref{eq:gdef} ensure that the leading coefficient
of each linear factor of $F$ divides $(\prod_{j=1}^d w_j^2)(\prod_{j=1}^{L} p_j^{s+1})$.
Thus, $p \mid w_k$ for some $1 \leq k \leq d$.
Our construction \eqref{eq:ca} of $c$ ensures that
$c\equiv w_k-\alpha_k \equiv -\alpha_k \pmod{w_k}$ and hence
\begin{equation}\label{eq:h0t}
h_0(t)\equiv c  \equiv -\alpha_k \pmod{p}.
\end{equation}
The construction of $w_k$ implies 
$\gcd(w_{k}, \beta_i-\alpha_{k})=1$ since no prime in
the set $\P'$ defined by \eqref{eq:PP} divides $w_k$.
Thus, for $i=1,2,\ldots,m$ and all $t \in \Z$,
\begin{equation*}\qquad
h_i(t) \equiv \beta_i+c \equiv \beta_i-\alpha_k \not\equiv 0 \pmod{p} .
\end{equation*}

Since $p \nmid w_i$ for $i \neq k$, \eqref{eq:h0t} implies that for all $t \in \Z$,
\begin{equation*}\qquad
g_i(t) = \frac{h_0(t)+\alpha_i}{w_i\prod_{j=1}^L p_j^{x_{i,j}}} \equiv \frac{-\alpha_k+\alpha_i}{w_i\prod_{j=1}^L p_j^{x_{i,j}}}  \not \equiv 0\pmod{p}
\end{equation*}
because $\gcd(w_i,\alpha_i-\alpha_{k})=1$ since no prime in
$\P'$ divides $w_i$.
Now consider the case $i=k$, for which $p \mid w_k$.  Then \eqref{eq:ca} ensures that
\begin{equation*}
\frac{c+\alpha_{k}}{w_{k}}\equiv 1\pmod{w_{k}} 
\qquad\text{and hence}\qquad \frac{c+\alpha_{k}}{w_{k}}\equiv 1 \pmod{p}.
\end{equation*}
For all $t \in \Z$, \eqref{eq:h0def} and \eqref{eq:gdef} imply
\begin{equation*}\qquad
g_k(t) \equiv \prod_{j=1}^L p_j^{-x_{i,j}}  \not\equiv 0 \pmod{p}.
\end{equation*}

Since no linear factor of $F$ vanishes identically, we have reached a contradiction.
Consequently,  $F$ does not vanish identically modulo any $p \notin \{p_1,p_2,\ldots,p_L\}$.

\subsection{Conclusion}\label{Subsection:Lem2}
Dickson's conjecture provides infinitely many $t$ such that 
\begin{align*}
&\text{$h_i(t)$ is prime for $1 \leq i \leq m$},\\
&\text{$g_j(t)$ is prime for $1 \leq j \leq d$},\\
&\text{$g_j(t) > \max\{w_j,p_L\}$ for $1 \leq j \leq d$}.
\end{align*}
Let $n = h_0(t)$ for any such $t$.  Then \eqref{eq:hdef} and \eqref{eq:gdef} imply
\begin{align*}
n+ \beta_i &= h_i(t)  && \text{for $1 \leq i \leq m$},\\
n+\alpha_i &= g_i(t)w_i\prod_{j=1}^L p_j^{x_{i,j}} && \text{for $1 \leq j \leq d$}.
\end{align*}
Since $g_i(t)$, $w_i$, and $\prod_{j=1}^L p_j^{x_{i,j}}$ are pairwise relatively prime
for each $1 \leq i \leq d$,
\begin{equation*}
f(n+\alpha_i) 
= f\bigg(g_i(t)w_i\prod_{j=1}^L p_j^{x_{i,j}}\bigg)
=f\big(g_i(t)\big)f(w_i)f\bigg(\prod_{j=1}^L p_j^{x_{i,j}}\bigg)
\end{equation*}
because $f$ is multiplicative.  
Condition (a) asserts that $\lim_{p\to \infty} f(p)=1$, so
$f(g_i(t)) = 1 + o(1)$ as $t$ increases.
By definition, each $f(w_i) \in I_i$, the open interval defined by \eqref{eq:I}
Consequently, if $t$ is sufficiently large
\begin{equation*}
f(n+\alpha_i) \in 
f\bigg(\prod_{j=1}^L p_j^{x_{i,j}} \bigg)I_i = \big(\xi_i(1-\epsilon),\xi_i(1+\epsilon)\big)
\end{equation*}
for $1 \leq i \leq d$.  Since $\epsilon > 0$ was arbitrary, we conclude that
\begin{equation*}\small
\Big\{ \big(f(n+\alpha_1), f(n+\alpha_2), \dots, f(n+\alpha_d) \big)
: n+\beta_1,n+\beta_2,\ldots, n+\beta_m \in \P \Big\} 
\end{equation*}
is dense in $[0,r]^d$.\qed

\section{Proof of Theorem \ref{Theorem:Main}}\label{Section:Proof}

Suppose that $f$ is a positive multiplicative function such that 
\begin{enumerate}
    \item $\lim_{p\to \infty} f(p)=1$, and
    \item $\prod_{p}f(p)$ is not absolutely convergent,
\end{enumerate}
$\alpha_1,\alpha_2,\ldots, \alpha_{d} \in \Z$ are distinct,
and $(\beta_1, \beta_2, \ldots, \beta_m)$ is an admissible $m$-tuple with $\alpha_i\neq \beta_j$
for all $i,j$.  Suppose that 
\begin{equation*}
\lim_{n\to\infty} \frac{h(n+1)}{h(n)} = \kappa \in (0,\infty)
\end{equation*}
and define $g = fh$.
Since
\begin{equation*}
\dfrac{g(n+\alpha_i)}{g(n+\alpha_j)}=\dfrac{f(n+\alpha_i)}{f(n+\alpha_j)}\cdot \dfrac{h(n+\alpha_i)}{h(n+\alpha_j)},
\end{equation*}
and
\begin{equation*}
\lim_{n\to \infty} \frac{h(n+\alpha_i)}{h(n+\alpha_j)} = \kappa^{\alpha_i - \alpha_j},
\end{equation*}
to prove the density of either \eqref{eq:S1} or \eqref{eq:S2} in $[0,\infty)^{d-1}$,
it suffices to consider the case in which $h$ is identically $1$.

Condition (b) ensures that there is an $\S \subseteq \P$ such that 
$\prod_{p \in \S}f(p)$ diverges to $0$ or $\infty$.  Assume Dickson's conjecture and
apply Theorem \ref{Theorem:Primary}
to $f$ or $1/f$, respectively, and conclude that 
\begin{equation}\label{eq:S}
S=
\big\{\big(f(n+\alpha_1), f(n+\alpha_2), \dots, f(n+\alpha_d)\big): 
n+\beta_1,n+\beta_2,\ldots, n+\beta_m \in \P \big\} 
\end{equation}
is dense in $[0,r]^d$ or $[r,\infty)^d$ for some $r > 0$.
By possibly considering $1/f$ in place of $f$, we may assume that
$S$ is dense in $[0,r]^d$.  Let $(\xi_1,\xi_2,\ldots,\xi_d) \in [0,r]^d$ and set
\begin{equation*}
\rho = \max \{\xi_1, \xi_2, \dots,\xi_d\}.
\end{equation*}
Given $\epsilon >0$, let $\delta>0$ be such that 
\begin{equation*}
1-\epsilon < \frac{1-\delta}{1+\delta}<\frac{1+\delta}{1-\delta}<1+\epsilon.
\end{equation*}
Select $x_1 \in (0,r/\rho)\cap (0,r)$ and define $x_i=x_1\xi_i$ for $2\leq i \leq d$. 
Thus,
\begin{equation*}
0< x_i =  x_1\xi_i <  \bigg(\frac{r}{\rho} \bigg)\rho = r
\quad \text{for $1 \leq i \leq d$}
\end{equation*}
and hence $( x_1,x_2, \ldots, x_{d}) \in (0,r)^{d}$.
Since $S$ is dense in  $[0,r]^{d}$, there is an $n\in \N$ such that 
$n+\beta_1,n+\beta_2,\ldots, n+\beta_m$ are prime and
\begin{equation*}
| f(n+\alpha_i) - x_i | < \delta x_i \quad\text{for $1 \leq i\leq d$}. 
\end{equation*}
Consequently, 
\begin{equation*}
\frac{f(n+\alpha_{i-1})}{f(n+\alpha_1)}<\dfrac{x_{i-1}(1+\delta)}{x_1(1-\delta)}=\xi_{i-1}\frac{1+\delta}{1-\delta}<\xi_{i-1}(1+\epsilon),
\end{equation*}
and
\begin{equation*}
\frac{f(n+\alpha_{i-1})}{f(n+\alpha_1)}>\dfrac{x_{i-1}(1-\delta)}{x_1(1+\delta)}=\xi_{i-1}\frac{1-\delta}{1+\delta}>\xi_{i-1}(1-\epsilon)
\end{equation*}
for $2 \leq i \leq d$.  In particular,
\begin{equation*}
\frac{f(n+\alpha_{i-1})}{f(n+\alpha_1)}\in \big((1-\epsilon)\xi_{i-1},(1+\epsilon)\xi_{i-1} \big)
\quad \text{for $i=2,3,\ldots, d-1$},
\end{equation*}
and hence the set \eqref{eq:S2} is dense in $[0,\infty)^{d-1}$.
Lemma \ref{Lemma:Homeo} provides the corresponding result for the set \eqref{eq:S1}. 
This completes the proof Theorem \ref{Theorem:Main}. \qed

\section{Further research}\label{Section:Further}

We have focused on primality constraints of the form
$n+ \beta_1, n+\beta_2, \ldots, n+\beta_m \in \P$
and simple shifts $n+\alpha_1, n+\alpha_2,\ldots, n +\alpha_d$ in the arguments
of the multiplicative function.
One can consider more general conditions.
For this, Dickson's conjecture (which concerns only linear polynomials) no longer suffices.  However,
the Bateman--Horn conjecture may permit such a generalization \cite{Bateman,Bateman2,1CRTA}.

\begin{problem}
Generalize Theorems \ref{Theorem:Main} and \ref{Theorem:Primary} to include polynomial primality constraints
$P_1(n), P_2(n),\ldots, P_m(n) \in \P$.
\end{problem}

\begin{problem}
Generalize Theorems \ref{Theorem:Main} and \ref{Theorem:Primary} 
so that the arguments $n+\alpha_1, n+\alpha_2,\ldots, n+\alpha_d$ are replaced by polynomial
functions of $n$.
\end{problem}

Obviously, it would be of interest to generalize in both directions simultaneously.
The interplay between the two conditions is likely to be nontrivial since already Theorem \ref{Theorem:Main}
requires that $\alpha_i \neq \beta_j$ for $1 \leq i \leq d$ and $1 \leq j \leq m$.  Example \ref{Example:Bad}
shows that this restriction is, at least in some cases, necessary.

\bibliographystyle{plain}
\bibliography{MSTT4MF}

\end{document}